\documentclass{article}
\usepackage{latexsym}
\usepackage{amssymb}
\usepackage{amsmath}
\usepackage{amsfonts}

\newtheorem{theorem}{Theorem}

\newtheorem{definition}[theorem]{Definition}
\newtheorem{example}[theorem]{Example}

\newtheorem{lemma}[theorem]{Lemma}

\newenvironment{proof}[1][Proof]{\noindent\textbf{#1.} }{\ \rule{0.5em}{0.5em}}

\begin{document}

\title{\textbf{Ballot Paths Avoiding Depth Zero Patterns}}
\author{{\Large Heinrich Niederhausen and Shaun Sullivan} \\
Florida Atlantic University, Boca Raton, Florida \\
\texttt{niederha@fau.edu, ssulli21@fau.edu}}
\date{}
\maketitle

\section{Introduction}

In a paper by Sapounakis, Tasoulas, and Tsikouras \cite{stt}, the authors
count the number of occurrences of patterns of length four in Dyck paths. In
this paper we specify in one direction and generalize in another. We only
count ballot paths that avoid a given pattern, where a ballot path stays
weakly above the diagonal $y=x$, starts at the origin, and takes steps from
the set $\{\uparrow ,\rightarrow \}=\{u,r\}$. A pattern is a finite string
made from the same step set; it is also a path. Notice that a ballot path
ending at a point along the diagonal is a Dyck path. Consider the following
enumeration of ballot paths avoiding the pattern $rur$.

\begin{equation*}
\begin{tabular}{l}
$%
\begin{tabular}{l||llllllllll}
$m$ & 1 & 8 & 28 & 62 & 105 & 148 & 178 & 178 & 127 & 0 \\ 
$7$ & 1 & 7 & 21 & 40 & 59 & 72 & 72 & 51 & 0 &  \\ 
$6$ & 1 & 6 & 15 & 24 & 30 & 30 & 21 & 0 &  &  \\ 
$5$ & 1 & 5 & 10 & 13 & 13 & 9 & 0 &  &  &  \\ 
$4$ & 1 & 4 & 6 & 6 & 4 & 0 &  &  &  &  \\ 
$3$ & 1 & 3 & 3 & 2 & 0 &  &  &  &  &  \\ 
$2$ & 1 & 2 & 1 & 0 &  &  &  &  &  &  \\ 
$1$ & 1 & 1 & 0 &  &  &  &  &  &  &  \\ 
$0$ & 1 & 0 &  &  &  &  &  &  &  &  \\ \hline\hline
& $0$ & $1$ & $2$ & $3$ & $4$ & $5$ & $6$ & $7$ & $8$ & $n$%
\end{tabular}%
$ \\ 
\multicolumn{1}{c}{The number of ballot paths to $\left( n,m\right) $
avoiding $rur$}%
\end{tabular}%
\end{equation*}

$s_n(m)$, the entry at the point $(n,m)$, is the number of ballot paths from
the origin to that point avoiding the pattern $rur$. We can generate this
table by using the recurrence formula 
\begin{equation}
s_n(m)=s_n(m-1)+s_{n-1}(m)-s_{n-1}(m-1)+s_{n-1}(m-2)
\label{rur}
\end{equation}
where $s_0(m)=1$, and $s_n(n-1)=0$ represents the weak boundary $y=x$. The
above table is calculated using this recursion and suggests the columns are
the values of a polynomial sequence. It will be shown that $s_n(m)$ is a
polynomial in $m$. We shall present a theorem guaranteeing the existence of
a polynomial sequence for a given recurrence relation and boundary
conditions.

If the number of ballot paths reaching $(n,x)$ is the values of a polynomial
sequence $s_{n}(x)$, the recurrence relation obtained can be transformed
into an operator equation. Using the tools from Finite Operator Calculus in
Section 2, we can find an explicit formula for the polynomial sequence $%
s_{n}(x)$.

In this paper, we will only consider patterns $p$ such that its reverse
pattern $\tilde{p}$ is a ballot path. For example, the reverse pattern of $%
p=uururrr$ is $\tilde{p}=uuururr$. We call such patterns \textit{depth-zero}
patterns. This is a requirement in Theorem \ref{ExistThm}, although for some
patterns that are not depth-zero, the solution may be a polynomial sequence,
but not for the entire enumeration. For patterns of nonzero depth see \cite%
{NieSul3}. To develop the recursions, we need to investigate the properties
of the pattern we wish to avoid.

\begin{definition}
A nonempty string $o\in\{u,r\}^*$ is a bifix of a pattern $p$ if it can be
written in the form $p=op^{\prime }$ and $p=p^{\prime \prime }o$ for
nonempty strings $p^{\prime },p^{\prime \prime }\in \{u,r\}^*$. If a pattern
has no bifixes, then we call it bifix-free.
\end{definition}

\begin{definition}
If $a$ is the number of $r$'s in $p$ and $c$ is the number of $u$'s, then we
say $p$ has dimensions $a\times c$, and $d(p)=a-c$.
\end{definition}

If $o$ is a bifix for the pattern $p$, then we will see in Section 4 that we
need that $d(p^{\prime \prime })\geq0$. The following lemma shows that this
restriction is not necessary when $p$ is depth zero.

\begin{lemma}
\label{Lem1}If $p$ is a depth zero pattern and $o$ is a bifix of $p$, then $%
d(p^{\prime \prime })\geq 0$.
\end{lemma}

\begin{proof}
First we note the additive property of $d$, that is $d(ab)=d(a)+d(b)$. By
definition $p=op^{\prime }=p^{\prime \prime }o$, which implies $d(p^{\prime
})=d(p^{\prime \prime })$. Since $p$ is depth zero, any suffix of $p$ is
also depth zero, in particular $p^{\prime }$ is depth zero. Therefore, $%
d(p^{\prime })\geq0$, which implies $d(p^{\prime \prime })\geq0$.
\end{proof}

\section{Main Tools}

In this section we will present the main tools from Finite Operator Calculus
that will be used to solve these enumeration problems. As we noted before,
not every pattern we choose to avoid will have a polynomial sequence
solution. The following theorem will illustrate why we want a depth-zero
pattern.

\begin{theorem}
\label{ExistThm}Let $x_{0},x_{1},...$ be a given sequence of integers where $%
F_{n}(m)=0$ for $m<x_{n}$. For $m>x_{n}$ define $F_{n}\left(
m\right) $ recursively for all $n\geq 0$ by%
\begin{equation*}
F_{n}(m)=F_{n}(m-1)+\sum\limits_{i=1}^{n}\sum\limits_{j}a_{i,j}F_{n-i}(m-b_{i,j}),
\end{equation*}%
where $\sum\limits_{j}a_{1,j}\neq 0$ and with initial values $F_{n}(x_{n})$.
If $F_{0}(m)=c\neq 0$ for all $m\geq x_{0}$, and if for each $j
$ holds $b_{i,j}\leq x_{n}-x_{n-i}+1$ for $n\geq 1$ and $i=1,\ldots ,n
$, then there exists a polynomial sequence $f_{n}(x)$ where $\deg {f_{n}}=n$
such that $F_{n}(m)=f_{n}(m)$ for all $m\geq x_{n}$ and $n\geq 0$ (within
the recursive domain).
\end{theorem}

Before we prove this theorem, we look back at the example of ballot paths
avoiding the pattern $rur$.

\begin{equation*}
\begin{tabular}{l}
$%
\begin{tabular}{l||llllllllll}
$m$ & 1 & 8 & 28 & 91 & 105 & 148 & 178 & 178 & 127 & \textbf{0} \\ 
$7$ & 1 & 7 & 21 & 62 & 59 & 72 & 72 & 51 & \textbf{0} &  \\ 
$6$ & 1 & 6 & 15 & 24 & 30 & 30 & 21 & \textbf{0} &  &  \\ 
$5$ & 1 & 5 & 10 & 13 & 13 & 9 & \textbf{0} &  &  &  \\ 
$4$ & 1 & 4 & 6 & 6 & 4 & \textbf{0} &  &  &  &  \\ 
$3$ & 1 & 3 & 3 & 2 & \textbf{0} &  &  &  &  &  \\ 
$2$ & 1 & 2 & 1 & \textbf{0} &  &  &  &  &  &  \\ 
$1$ & 1 & 1 & \textbf{0} &  &  &  &  &  &  &  \\ 
$0$ & 1 & \textbf{0} & 0 &  &  &  &  &  &  &  \\ 
$-1$ & \textbf{1} & -1 & -1 &  &  &  &  &  &  &  \\ \hline\hline
& $0$ & $1$ & $2$ & $3$ & $4$ & $5$ & $6$ & $7$ & $8$ & $n$%
\end{tabular}%
$ \\ 
$F_{n}(m)=F_{n}(m-1)+F_{n-1}(m)-F_{n-1}(m-1)+F_{n-1}(m-2)$%
\end{tabular}%
\end{equation*}

Here, $x_n=n-1$ and $F_n(x_n)=\delta_{n,0}$. The recursion satisfies the
conditions $b_{i,j}\leq x_n-x_{n-i}+1=i+1$ and $\sum\limits_ja_{1,j}=1\neq0$%
. In the table, below the initial values are the polynomial extensions,
which are not all 0. Clearly, if the recursion required us to use values
below these initial values, the values generated would not agree with the
values generated by the polynomial extensions, thus we need the condition $%
b_{i,j}\leq x_n-x_{n-i}+1$, in general. We now prove Theorem \ref{ExistThm}.

\begin{proof}
We first let $f_0(x)=F_0(x_0)$, since $F_0(m)=F_0(m-1)$. Now suppose $F_n(m)$
has been extended to a polynomial of degree $n$. The recursion tells us that 
$F_{n+1}(m)-F_{n+1}(m-1)$ can be extended to a polynomial of degree $n$
since $\sum\limits_ja_{1,j}\neq0$. But if the backwards difference of a
function is a polynomial, then the function itself is a polynomial of one
degree higher (see \cite{Jordan}).
\end{proof}

The objects of Finite Operator Calculus are polynomial sequences called
Sheffer sequences. A Sheffer sequence is defined by its generating function,
which can be written in the following way.%
\begin{equation*}
p(x,t)=\sum\limits_{n\geq 0}s_{n}(x)t^{n}=\sigma (t)e^{x\beta (t)}
\end{equation*}%
where $\sigma (t)$ is a power series with a multiplicative inverse and $%
\beta (t)$ is a delta series, that is, a power series with a compositional
inverse. We say $(s_{n})$ is a Sheffer sequence for the basic sequence $%
(b_{n})$ where%
\begin{equation*}
b(x,t)=\sum\limits_{n\geq 0}b_{n}(x)t^{n}=e^{x\beta (t)}.
\end{equation*}%
Notice that $b_{n}(0)=\delta _{n,0}$ and $b_{0}(x)=1$.

In order to find the solution to the recurrence, we first transform the
recursion into an equation with finite operators. For every Sheffer
sequence, the linear operator $B=\beta ^{-1}(D)$, where $D=\frac{d}{dx}$,
maps $s_{n}$ to $s_{n-1}$. The operator $B$ also maps its basic sequence $%
b_{n}$ to $b_{n-1}$ (Rota, Kahaner, and Odlyzko \cite{ROK}). Also if $B$ is
a \textit{delta} operator, that is

\begin{equation*}
B=\sum\limits_{n\geq 1}a_{n}D^{n}
\end{equation*}%
where $a_{1}\neq 0$, then it corresponds to a class of Sheffer sequences
associated to a unique basic sequence. We do not have to worry about
convergence since these operators are only applied to polynomials, thus only
a finite number of the terms will be used for a given polynomial, and thus
the name Finite Operator Calculus.

An example of a delta operator that will be frequently used in this paper is 
$\nabla=1-E^{-1}$, where $E^as_n(x)=s_n(x+a)$ is a shift operator. By
Taylor's Theorem, we can write

\begin{equation*}
f(x+a)=\sum\limits_{n\geq0}f^{(n)}(x)\frac{a^n}{n!} \qquad or \qquad
E^a=\sum\limits_{n\geq0}\frac{D^na^n}{n!}=e^{aD}
\end{equation*}
in operators. Clearly $\nabla$ is a delta operator.

If we consider $s_n(x)$ for $x\geq n-1$ in (\ref{rur}), we obtain the
operator equation

\begin{equation*}
\nabla=B-BE^{-1}+BE^{-2}.
\end{equation*}
We have written one delta operator in terms of an unknown delta operator.
The following theorem \cite[Theorem 1]{Nied03} lets us find the basic
sequence for the unknown operator $B$ in terms of the basic sequence of the
known operator.

\begin{theorem}[Tranfer Formula]
Let $A$ be a delta operator with basic sequence $a_n(x)$. Suppose

\begin{equation*}
A=\tau (B)=\sum\limits_{j\geq 1}T_{j}B^{j},
\end{equation*}%
where the $T_{j}$ are linear operators that commute with shift operators,
and $T_{1}$ is invertible. Then $B$ is a delta operator with basic sequence

\begin{equation*}
b_n(x)=x\sum\limits_{i=1}^n\left[\tau^i\right]_n\frac{1}{x}a_i(x).
\end{equation*}
where $\left[\tau^i\right]_n$ is the coefficient of $t^n$ in $\tau^i(t)$.
\end{theorem}

Notice that this Theorem implies that the unknown operator $B$ is a delta
operator, and so the underlying polynomial sequence is a Sheffer sequence.
The final step is to transform this basic sequence into the correct Sheffer
sequence using the initial values. For ballot paths we have that $%
s_n(n-1)=\delta_{n,0}$, thus the following lemma gives us the solution based
on these initial values.

\begin{lemma}
If $(b_n)$ is a basic sequence for $B$, then

\begin{equation*}
t_n(x)=(x-an-c)\dfrac{b_n(x-c)}{x-c}
\end{equation*}
is the Sheffer sequence for $B$ with initial values $t_n(an+c)=\delta_{n,0}$.
\end{lemma}

In our case, we would obtain the solution from the basic sequence $b_n$
given in the transfer formula as

\begin{equation}
s_{n}(x)=(x-n+1)\dfrac{b_{n}(x+1)}{x+1}
\label{abel}
\end{equation}%
called Abelization \cite[(2.5)]{Nied03}. We want to remark that given
the generating function $e^{x\beta \left( t\right) }=\sum_{n\geq
0}b_{n}\left( x\right) t^{n}$, we find 
\begin{equation*}
\sum_{n\geq 0}s_{n}\left( x\right) t^{n}=e^{\left( x+1\right) \beta \left(
t\right) }-\frac{t}{x+1}\frac{d}{dt}e^{(x+1)\beta \left( t\right) }=\left(
1-t\beta ^{\prime }\left( t\right) \right) e^{\left( x+1\right) \beta \left(
t\right) }.
\end{equation*}%
All the examples connected to pattern avoidance are of the form
\begin{equation*}
1-E^{-1}=\nabla =\sum_{k\geq 1}a_{k}E^{c_{k}}B^{k}
\end{equation*}%
where $a_{k},c_{k}\in \mathbb{Z}$. If this equation can be solved for $E$,
\begin{equation*}
E=1+\tau \left( B\right) ,
\end{equation*}%
where $\tau$ is a delta series, then
\begin{equation*}
\sum_{n\geq 0}b_{n}\left( x\right) t^{n}=\left( 1+\tau \left( t\right)
\right) ^{x}
\end{equation*}

We first present the results in \cite{NieSul3} which covered bifix-free
patterns and patterns with exactly one bifix. We then present results for
patterns with any number of bifixes, hence for any depth-zero pattern. We
also give some examples and special cases.

\section{Bifix-free patterns}

Let $s_{n}\left( x;p\right) $ be the number of ballot paths avoiding the
pattern $p$; we occasionally will drop the $p$ in the notation when
convenient. If the pattern is bifix-free and depth-zero, we need only to
subtract paths that would end in the pattern, thus we have the recurrence%
\begin{equation}
s_{n}(m;p)=s_{n-1}(m;p)+s_{n}(m-1;p)-s_{n-a}(m-c;p)  \label{(bifix0)}
\end{equation}%
where $p$ has dimensions $a\times c$. For example $uurrurrur$ has dimensions 
$5\times 4$, and depth zero. If the pattern has depth zero,\ and $a\geq
c\geq 1$, $a\geq 2$ (if $a=1$ then $p=ur$, a pattern we do not want to
avoid), then it is easy to check that the conditions of the Theorem \ref%
{ExistThm} are satisfied; the solution is polynomial. In operator notation,%
\begin{equation*}
\nabla =B(1-B^{a-1}E^{-c}).
\end{equation*}

Since the delta operator $\nabla $ can be written as a delta series in $B$,
the operator $B$ is also a delta operator. The basic sequence can be
expressed via the Transfer Formula%
\begin{equation*}
b_{n}(x)=x\sum\limits_{i=0}^{\frac{n}{a-1}}\frac{(-1)^{i}}{x-ci}\dbinom{%
n-(a-1)i}{i}\dbinom{x+n-(a+c-1)i-1}{n-(a-1)i}.
\end{equation*}

Using (\ref{abel}) we obtain%
\begin{equation*}
s_{n}(x)=(x-n+1)\sum\limits_{i=0}^{\frac{n}{a-1}}\frac{(-1)^{i}}{x-ci+1}%
\dbinom{n-(a-1)i}{i}\dbinom{x+n-(a+c-1)i}{n-(a-1)i}.
\end{equation*}

Therefore the number of ballot paths avoiding $p$ and returning to the
diagonal (Dyck paths) equals%
\begin{equation*}
s_{n}(n)=\sum\limits_{i=0}^{\frac{n}{a-1}}\frac{(-1)^{i}}{n-ci+1}\dbinom{%
n-(a-1)i}{i}\dbinom{2n-(a+c-1)i}{n-(a-1)i}
\end{equation*}

\section{Patterns with exactly one bifix}

If the pattern $p$ has exactly one bifix, then there exists a unique
nonempty pattern $o$ such that $p=op^{\prime }o$. If $p$ has depth $0$ and
dimensions $a\times c,$ and $p^{\prime \prime }$ has dimensions $b\times d$,
we have a recurrence of the form%
\begin{equation}
s_{n}(x)=s_{n-1}(x)+s_{n}(x-1)-\sum\limits_{i\geq
0}(-1)^{i}s_{n-a-bi}(x-c-di).  \label{(bifix1)}
\end{equation}

For example let $p=urruurr$. This pattern is depth-zero, with dimensions $%
4\times 3$ and bifix $o=urr$, so $b=2$ and $d=2$. From the paths reaching $%
(n-a,x-c)$, those ending in $p^{\prime \prime }=urru$ cannot be included in
the recurrence and must be subtracted, and from those again we cannot
include paths ending in $urru$, and so on. The $p^{\prime \prime }$ piece of
the pattern that is responsible for this exclusion-inclusion process may not
go below the diagonal; hence it must have $d(p^{\prime \prime })\geq 0$
which lemma \ref{Lem1} guarantees. If $d(p^{\prime \prime })<0$, then at
some point we would be using numbers below the $y=x$ boundary, which are
only the polynomial extensions and do not count paths.

In operators, we have%
\begin{equation*}
1=B+E^{-1}-\sum\limits_{i\geq 0}(-1)^{i}B^{a+bi}E^{-c-di}=B+E^{-1}-\dfrac{%
B^{a}E^{-c}}{1+B^{b}E^{-d}}
\end{equation*}

or%
\begin{eqnarray}
\nabla &=&B+B^{b+1}E^{-d}+B^{b}E^{-d-1}-B^{a}E^{-c}-B^{b}E^{-d}
\label{(index1)} \\
&=&B-B^{a}E^{-c}+B^{b+1}E^{-d}-B^{b}E^{-d}\nabla .  \notag
\end{eqnarray}

Using the Transfer Formula, we obtain $b_{n}(x)=$%
\begin{eqnarray*}
&&x\sum_{j,k,l\geq 0}\binom{n-(a-1)j-bk-(b-1)l}{j,k,l}\frac{(-1)^{j+l}}{%
n-(a-1)j-bk-(b-1)l} \\
&&\times \binom{n-(a+c-1)j-(b+d)(k+l)-1+x}{n-(a-1)j-b(k+l)-1}.
\end{eqnarray*}%
Because $s_{n}(n-1)=\delta _{0,n}$ we get $s_{n}(x)=$%
\begin{eqnarray*}
&&\sum_{j,k,l\geq 0}\binom{n-(a-1)j-bk-(b-1)l}{j,k,l}\frac{(-1)^{j+l}(x-n+1)%
}{n-(a-1)j-bk-(b-1)l} \\
&&\times \binom{n-(a+c-1)j-(b+d)(k+l)+x}{n-(a-1)j-b(k+l)-1}.
\end{eqnarray*}

So for the pattern $urruurr$, we obtain $s_{n}(x)=$%
\begin{equation*}
\sum_{j,k,l\geq 0}\binom{n-3j-2k-l}{j,k,l}\dfrac{(-1)^{j+l}(x-n+1)}{n-3j-2k-l%
}\binom{n-6j-4(k+l)+x}{n-3j-2(k+l)-1}.
\end{equation*}

\begin{equation*}
\begin{tabular}{l}
$%
\begin{tabular}{l||llllllllll}
$m$ & 1 & 8 & 35 & 110 & 270 & 544 & 920 & 1272 & 1236 & 0 \\ 
$7$ & 1 & 7 & 27 & 75 & 161 & 279 & 389 & 377 & 0 &  \\ 
$6$ & 1 & 6 & 20 & 48 & 87 & 122 & 118 & 0 &  &  \\ 
$5$ & 1 & 5 & 14 & 28 & 40 & 38 & 0 &  &  &  \\ 
$4$ & 1 & 4 & 9 & 14 & 13 & 0 &  &  &  &  \\ 
$3$ & 1 & 3 & 5 & 5 & 0 &  &  &  &  &  \\ 
$2$ & 1 & 2 & 2 & 0 &  &  &  &  &  &  \\ 
$1$ & 1 & 1 & 0 &  &  &  &  &  &  &  \\ 
$0$ & 1 & 0 &  &  &  &  &  &  &  &  \\ \hline\hline
& $0$ & $1$ & $2$ & $3$ & $4$ & $5$ & $6$ & $7$ & $8$ & $n$%
\end{tabular}%
$ \\ 
\multicolumn{1}{c}{The number of ballot paths avoiding $urruurr$}%
\end{tabular}%
\end{equation*}

\section{A Special Case}

The above operator equation (\ref{(index1)}) simplifies when $a=b+1$ and $%
c=d $. This corresponds to a pattern of the form $rp^{\prime }r$.\emph{\ }%
For this case we get 
\begin{equation*}
\nabla =B(1+B^{b}E^{-d})^{-1},
\end{equation*}

\begin{equation*}
b_{n}(x)=x\sum\limits_{i\geq 0}\dfrac{(-1)^{i}}{x-di}\dbinom{n-(b-1)i-1}{i}%
\dbinom{x+n-(d+b)i-1}{n-bi},
\end{equation*}%
and%
\begin{equation*}
s_{n}(x)=(x-n+1)\sum\limits_{i\geq 0}\dfrac{(-1)^{i}}{x-di+1}\dbinom{%
n-(b-1)i-1}{i}\dbinom{x+n-(d+b)i}{n-bi}.
\end{equation*}

The number of Dyck paths equals%
\begin{equation*}
s_{n}(n)=\sum\limits_{i\geq 0}\dfrac{(-1)^{i}}{n-di+1}\dbinom{n-(b-1)i-1}{i}%
\dbinom{2n-(d+b)i}{n-bi}.
\end{equation*}

The first pattern we considered, $rur$, is of this form where $b=d=1$. The
solution in this case is%
\begin{equation*}
s_{n}(x)=(x-n+1)\sum\limits_{i\geq 0}\dfrac{(-1)^{i}}{x-i+1}\dbinom{n-1}{i}%
\dbinom{x+n-2i}{n-i}
\end{equation*}

and the number of Dyck paths equals%
\begin{equation*}
s_{n}(n)=\sum\limits_{i\geq 0}\dfrac{(-1)^{i}}{n-i+1}\dbinom{n-1}{i}\dbinom{%
2n-2i}{n-i}=\sum\limits_{i\geq 0}(-1)^{i}C_{n-i}\dbinom{n-1}{i}.
\end{equation*}
where $C_{n}$ is the $n$th Catalan number.

\section{Patterns with any number of bifixes}

Now we suppose that the pattern $p$ has $m$ distinct bifixes $o_{1},\ldots
,o_{m}$. Let $b_{i}\times d_{i}$ be the dimensions of $p_{i}^{\prime \prime }
$, the pattern $p$ without the suffix $o_{i}$, and $a\times c$ be the
dimensions of $p$. The number of ways to reach the lattice point $(n-a-\sum
i_{j}b_{j},x-c-\sum i_{j}d_{j})$ with $k\geq 0$ combinations of the $%
p_{j}^{\prime \prime }$ with $\sum_{j=0}^{m}i_{j}=k$ is $\dbinom{k}{%
i_{1},\ldots ,i_{m}}$ with sign $(-1)^{k}$. Then the recurrence for the
ballot paths avoiding $p$ is 
\begin{eqnarray}
s_{n}(x) &=&s_{n-1}(x)+s_{n}(x-1)-\sum\limits_{k\geq 0}(-1)^{k} \\
&&\times \sum\limits_{i_{1}+\cdots +i_{m}=k}\dbinom{k}{i_{1},\ldots ,i_{m}}%
s_{n-a-\Sigma ~i_{j}b_{j}}(x-c-\sum i_{j}d_{j}).  \notag
\end{eqnarray}

In operators,%
\begin{eqnarray}
\nabla &=&B-\sum\limits_{k\geq 0}(-1)^{k}\sum\limits_{i_{1}+\cdots +i_{m}=k}%
\dbinom{k}{i_{1},\ldots ,i_{m}}B^{a+\sum i_{j}b_{j}}E^{-c-\sum i_{j}d_{j}} 
\notag \\
&=&B-B^{a}E^{-c}\sum\limits_{k\geq 0}(-1)^{k}\left(
\sum\limits_{i=0}^{m}B^{b_{i}}E^{-d_{i}}\right) ^{k}  \notag \\
&=&B-\dfrac{B^{a}E^{-c}}{1+\sum\limits_{i=0}^{m}B^{b_{i}}E^{-d_{i}}}
\label{(i)}
\end{eqnarray}

In general, we can find the basic sequence for any depth zero pattern via
the transfer formula, although it would involve many summations. We now
consider an example and some special cases.

\begin{example}
The pattern $rururrur$ has two bifixes, $rur$ and $r$. In this case, $a=5$, $%
c=3$, $b_{1}=3$, $d_{1}=2$, $b_{2}=4$ and $d_{2}=3$. Using the above
operator equation, we obtain%
\begin{eqnarray}
\nabla &=&B-\dfrac{B^{5}E^{-3}}{1+B^{3}E^{-2}+B^{4}E^{-3}}  \notag \\
&=&\dfrac{B+B^{4}E^{-3}}{1+B^{3}E^{-2}+B^{4}E^{-3}}  \notag
\end{eqnarray}

Using the transfer formula and (\ref{abel}),%
\begin{eqnarray}
b_{n}(x) &=&\sum\limits_{i=0}^{n}x\sum\limits_{j\geq 0}\sum\limits_{k\geq 0}%
\dbinom{i}{j}\dbinom{-i}{k}\dbinom{k}{n-i-3j-3k}\dfrac{1}{x-n+i+j+k}  \notag
\\
&&\times \dbinom{x-n+2i+j+k-1}{i}  \notag
\end{eqnarray}%
and%
\begin{eqnarray}
s_{n}(x) &=&(x-n+1)\sum\limits_{i=0}^{n}\sum\limits_{j\geq
0}\sum\limits_{k\geq 0}\dbinom{i}{j}\dbinom{-i}{k}\dbinom{k}{n-i-3j-3k} 
\notag \\
&&\times \dfrac{1}{x-n+i+j+k+1}\dbinom{x-n+2i+j+k}{i}.  \notag
\end{eqnarray}
\end{example}

\subsection{Patterns of the form $r^a$}

We had previously determined in \cite{N-Su07} by other methods that the
recurrence formula for this pattern is%
\begin{equation*}
s_{n}(x)=s_{n-1}(x)+s_{n}(x-1)-s_{n-a}(x-1)
\end{equation*}%
giving the operator equation%
\begin{equation*}
\nabla =B-B^{a}E^{-1}\qquad or\qquad E=\sum\limits_{i=0}^{a-1}B^{i}.
\end{equation*}

Using the new formula (\ref{(i)}), we view $r^{a}$ as a pattern with $a-1$
bifixes. Here $c=d_{i}=0$ and $b_{i}=i$ for all $0\leq i\leq a-1$, thus the
operator equation becomes%
\begin{equation}
\nabla =B-\dfrac{B^{a}}{1+\sum\limits_{i=1}^{a-1}B^{i}}=1-\dfrac{1}{%
\sum\limits_{i=0}^{a-1}B^{i}}\qquad or\qquad E=\sum\limits_{i=0}^{a-1}B^{i} 
\notag
\end{equation}%
as expected.

\subsection{Patterns of the form $r(ur)^k$}

If the pattern is of the form $r(ur)^{k}$, then $a=k+1$, $c=k$, and for each 
$0\leq i<k$ there is the bifix $r(ur)^{i}$. So, $b_{j}=d_{j}=j$ for $1\leq
j\leq k$. Let $S=BE^{-1}$. We obtain the following operator equation for
this pattern,%
\begin{eqnarray}
\nabla &=&B-\dfrac{BS^{k}}{1+S+S^{2}+\cdots +S^{k}}  \notag \\
\nabla &=&B\left( 1-\dfrac{S^{k}}{1+S+S^{2}+\cdots +S^{k}}\right) .  \notag
\end{eqnarray}

Using the Transfer Formula,%
\begin{equation*}
b_{n}(x)=\sum\limits_{i=0}^{n}x\sum\limits_{j\geq 0}(-1)^{j}\dbinom{i}{j}%
\dbinom{-j}{n-i-kj}_{k+1}\dfrac{1}{x-n+i}\dbinom{x-n+2i-1}{i}
\end{equation*}%
where $\dbinom{x}{n}_{k+1}:=[t^{n}](1+t+\cdots +t^{k})^{x}$ \cite{N-Su07}.
The number of Ballot paths is $s_{n}(x)=$

\begin{equation*}
(x-n+1)\sum\limits_{i=0}^{n}\sum\limits_{j\geq 0}(-1)^{j}\dbinom{i}{j}%
\dbinom{-j}{n-i-kj}_{k+1}\dfrac{1}{x-n+i+1}\dbinom{x-n+2i}{i}.
\end{equation*}%
We obtain a very nice formula for the Dyck paths in this case,%
\begin{equation*}
s_{n}(n)=\sum\limits_{i=0}^{n}C_{i}\sum\limits_{j\geq 0}(-1)^{j}\dbinom{i}{j}%
\dbinom{-j}{n-i-kj}_{k+1}
\end{equation*}%
where $C_{n}$ is the $n$th Catalan number.

\end{document}